\documentclass[11pt]{amsart}
\usepackage{tabularx,booktabs}
\usepackage{caption}
\usepackage{amsmath,tikz}
\usepackage{amsfonts,mathscinet}

\usepackage{amscd}
\usepackage{amsthm}
\usepackage{amssymb} \usepackage{latexsym}
\usepackage{eufrak}
\usepackage{euscript}
\usepackage{epsfig}
\usepackage{graphics}
\usepackage{array}
\usepackage{enumerate}
\usepackage{dsfont}
\usepackage{color}
\usepackage{wasysym}
\usepackage{hyperref}
\usepackage{pdfsync}

\newcommand{\bel}[1]{\begin{equation}\label{#1}}

\newcommand{\be}{\begin{equation}}

\newcommand{\ba}{\begin{eqnarray}}
\newcommand{\ea}{\end{eqnarray}}

\newcommand{\qe}{\end{equation}}
\newcommand{\R}{{\mathbb R}}

\newcommand{\SP}{\mathbb{S}}
\newcommand{\wt}{\widetilde}

\newcommand{\Al}{{\textit{Alex}(1)}}

\newcommand{\pt}{\mathrm{Pttn}}

\newcommand{\area}{\mathrm{Area}}

\newcommand{\cri}{\mathrm{cri}}

\newcommand{\PC}{\mathcal{PC}_{>0}}
\newcommand{\ST}{\mathcal{T}_{\SP^2}}
\newcommand{\ad}{\mathcal{A}{\mathrm{d}}\mathcal{P}}

\newcommand{\PCS}{\mathcal{PC}_{\geq 1}}

\newcommand{\Hmm}[1]{\leavevmode{\marginpar{\tiny%
$\hbox to 0mm{\hspace*{-0.5mm}$\leftarrow$\hss}%
\vcenter{\vrule depth 0.1mm height 0.1mm width \the\marginparwidth}%
\hbox to
0mm{\hss$\rightarrow$\hspace*{-0.5mm}}$\\\relax\raggedright #1}}}

\newtheorem{theorem}{Theorem}[section]

\newtheorem{lemma}[theorem]{Lemma}
\newtheorem{corollary}[theorem]{Corollary}
\newtheorem{definition}[theorem]{Definition}

\newtheorem{remark}[theorem]{Remark}

\newtheorem{proposition}[theorem]{Proposition}
\newtheorem{problem}[theorem]{Problem}

\newtheorem{assumption}{Assumption}

\begin{document}

\title[Areas of spherical polyhedral surfaces with regular faces]{Areas of spherical polyhedral surfaces with regular faces}

\author{Yohji Akama}
\email{akama@math.tohoku.ac.jp}
\address{Yohji Akama: Mathematical Institute, Graduate School of Science, Tohoku University,
Sendai, 980-0845, Japan.  \\ Tel.: +81-22-795-6402, Fax: +81-22-795-6400}

\author{Bobo Hua}
\email{bobohua@fudan.edu.cn}
\address{Bobo Hua: School of Mathematical Sciences, LMNS,
Fudan University, Shanghai 200433, China; Shanghai Center for
Mathematical Sciences, Fudan University, Shanghai 200433,
China.}

\author{Yanhui Su}
\email{suyh@fzu.edu.cn}
\address{Yanhui Su: College of Mathematics and Computer Science, Fuzhou University, Fuzhou
 350116, China}

\begin{abstract} 

 For a finite planar graph, it associates with some metric spaces, called (regular)
spherical polyhedral surfaces, by replacing faces with regular spherical
polygons in the unit sphere and gluing them edge-to-edge. We consider the class of planar graphs which admit
spherical polyhedral surfaces with the curvature bounded below by 1 in the
sense of Alexandrov, i.e. the total angle at each vertex is at most $2\pi$. We classify
all spherical tilings with regular spherical polygons, i.e. total angles at
vertices are exactly $2\pi$. We prove that for any graph in this class which does
not admit a spherical tiling, the area of the associated spherical polyhedral
surface with the curvature bounded below by 1 is at most $4\pi-\epsilon_0$ for
some $\epsilon_0>0$. That is, we obtain a definite gap between the area of
such a surface and that of the unit sphere.
 
 
 \end{abstract}

\maketitle



\par
\maketitle

\bigskip


\section{Introduction}\label{sec:intro}
The combinatorial curvature for planar graphs, as the generalization of the Gaussian curvature for surfaces,  was introduced by \cite{MR0279280,MR0410602,MR919829,Ishida90}. Many interesting geometric and combinatorial results have been obtained since then, see e.g.
\cite{MR1600371,Woess98,MR1864922,MR1797301,MR1894115,MR1923955,MR2038013,MR2096789,RBK05,MR2243299,MR2299456,MR2410938,MR2466966,MR2470818,MR2558886,MR2818734,MR2826967,MR3624614,Gh17}.

  Let $(V,E)$ be an
undirected simple graph with the set of vertices $V$ and the set of
edges $E.$ The graph $(V,E)$ is called \emph{planar} if it is topologically
embedded into the sphere or the plane. We write $G=(V,E,F)$ for the
combinatorial structure, or the cell complex, induced by the embedding
where $F$ is the set of faces, i.e. connected components of the
complement of the embedding image of the graph $(V,E)$ in the
target. Two elements in $V,E,F$ are called \emph{incident} if the
closures of their images have non-empty intersection. We say that a
planar graph $G$ is a \emph{planar tessellation} if the following hold,
see e.g. \cite{MR2826967}:
\begin{enumerate}[(i)]
\item Every face is homeomorphic to a disk whose boundary consists of finitely many edges of the graph.
\item Every edge is contained in exactly two different faces.
\item For any two faces whose closures have non-empty intersection, the intersection is either a vertex or an edge.
\end{enumerate}
\begin{assumption}
 We only consider planar tessellations and call them \emph{planar
 graphs} for the sake of simplicity.  For a planar graph, we always
 assume that for any vertex $x$ and face $\sigma,$ $$\deg(x)\geq 3,\
 \mathrm{deg}(\sigma)\geq 3$$ where $\deg(\cdot)$ denotes the degree of
 a vertex or a face.
 \end{assumption}
For a planar graph $G$, the \emph{combinatorial curvature} at the
vertex is defined as
\begin{equation}\label{def:comb}\Phi(x)=1-\frac{\deg(x)}{2}+\sum_{\sigma\in
F:x\in \overline{\sigma}}\frac{1}{\deg(\sigma)},\quad x\in
V,\end{equation} where the summation is taken over all faces $\sigma$
incident to $x.$ To digest the definition, we endow the ambient space of
$G$, $\SP^2$ or $\R^2,$ with a canonical piecewise flat metric and call
it the (regular) \emph{Euclidean polyhedral surface}, denoted by $S(G)$:
Replace each face by a regular Euclidean polygon of side-length one and
same facial degree, glue them together along their common edges, and
define the metric on $\SP^2$ or $\R^2$ via gluing metrics, see
\cite[Chapter~3]{MR1835418}. It is well-known that the generalized
Gaussian curvature on an Euclidean polyhedral surface, as a measure,
concentrates on the vertices. And one is ready to see that the
combinatorial curvature at a vertex is in fact the mass of the
generalized Gaussian curvature at that vertex up to the normalization
$2\pi,$ see e.g. \cite{MR2127379,MR3318509}.

We denote by $$\PC:=\{G=(V,E,F): \Phi(x)>0,\forall x\in V\}$$ the class
of planar graphs with positive combinatorial curvature everywhere. There
are many examples in $\PC,$ e.g. Platonic solids, Archimedean solids,
and Johnson solids~\cite{MR0185507,MR0227860,Zalgaller69}, and see \cite{RBK05,MR2836763,Gh17} for more. The
complete classification of $\PC$ is not yet known.
Note that by Alexandrov's embedding theorem~\cite{MR2127379}, the Euclidean polyhedral surface $S(G)$ of a finite planar graph $G\in \PC$ can be isometrically embedded into $\R^3$ as a boundary of a convex polyhedron. 

We review some known results on the class $\PC.$ 
In~\cite{MR2299456}, DeVos and Mohar proved that any graph $G\in \PC$ is
finite, which solves a conjecture of Higuchi~\cite{MR1864922}, see
\cite{MR0410602,MR2096789} for early results. In fact, for the set
\begin{align*}
 P:=\{G\colon G\in \PC\mbox{ is neither a prism nor an antiprism}\;\},
\end{align*}
DeVos and Mohar proved that $\sharp P<\infty$ and proposed to determine the number 
\begin{align*}C_{\SP^2}:=\max_{ (V,E,F) \in P}\sharp V.
\end{align*}
On the one hand, for the lower bound estimate of $C_{\SP^2}$ many
authors \cite{RBK05,MR2836763,Gh17,Ol17} attempted to construct large
examples in this class, and finally found some examples possessing $208$
vertices.  On the other hand, in~\cite{MR2299456}, DeVos and Mohar showed
that $C_{\SP^2}\leq 3444,$ which was improved to $C_{\SP^2}\leq 380$ by
Oh~\cite{MR3624614}. By a refined argument, in~\cite{Gh17}, Ghidelli completely solved the problem.
  \begin{theorem}\label{thm:208}
   \begin{enumerate}
\item \label{assert:208} $C_{\SP^2}=208.$ 
      \cite{Gh17,MR2836763}
      
  \item
\label{thm:41upper} 
$\max\{ \deg(\sigma)\colon \sigma\in F, (V,E,F)\in P\}\leq 41.$  \cite{Gh17}
\end{enumerate}\end{theorem}

In this paper, we study spherical polyhedral surfaces with faces
isometric to regular spherical polygons. Spherical polyhedral surfaces have been extensively studied in the literature, see e.g. \cite{MR0029518, MR0063681, LuoTian92, MR1835418, Luo06, Luo14}. A \emph{regular spherical polygon} is
the domain in the unit sphere $\SP^2(1)$ defined by the intersection of
finitely many hemispheres such that all side lengths and interior angles
are equal respectively.
\begin{assumption}\label{ass:hemisphere}We only consider regular spherical polygons
in $\SP^2(1)$ which have at least
three sides.\end{assumption} For any $n\geq 3$ and $0<a\leq \frac{2\pi}{n},$ there is a
regular spherical $n$-gon in $\SP^2(1)$ of side length $a$ contained in
a hemisphere, denoted by $\Delta_n(a),$ which is unique up to the
spherical isometry. 

Analogous to Euclidean polyhedral surfaces, we define spherical
polyhedral surfaces associated to a planar graph. For any finite planar
graph $G=(V,E,F)$ and $0<a\leq \inf_{\sigma\in
F}\frac{2\pi}{\deg(\sigma)},$ we replace each face by a regular
spherical polygon in $\SP^2(1)$ of side length $a,$ and glue them
together along their common edges. This induces a metric structure on
the ambient space of $G,$ called \emph{spherical polyhedral surface} of
$G$ with side length $a$ and denoted by $S_a(G).$ We denote by
$\area_a(G)$ the area of $S_a(G).$ For any $x\in V,$ the total angle at
$x$ measured in $S_a(G)$ is denoted by $\theta_a(x),$ and the angle
defect at $x$ is defined as $$K_a(x)=2\pi-\theta_a(x).$$ By the
spherical geometry, we yield the Gauss-Bonnet theorem on spherical
polyhedral surfaces.
\begin{theorem}\label{thm:GaussBonnet} For any finite planar graph
 $G=(V,E,F)$ and for any positive number $a\leq \inf_{\sigma\in F}\frac{2\pi}{\deg(\sigma)},$
\begin{equation}\label{GaussBonnet}\area_a(G)+\sum_{x\in V}K_a(x)=4\pi.\end{equation}
\end{theorem}

We say that a geodesic metric space $(X,d)$ has the (sectional)
curvature bounded below by $1$ in the sense of Alexandrov if it
satisfies the Toponogov triangle comparison property with respect to  the unit
sphere, and denote by $\Al$ the set of such spaces, see
e.g. \cite{MR1835418}.  It is well-known that $S_a(G)\in \Al$ if and
only if $$K_a(x)\geq 0,\quad \forall x\in V.$$ By Alexandrov and
Pogorelov's theorem, they can be embedded into $\SP^3$ as the boundary of
convex polyhedron~\cite{MR0029518,MR0063681}. In fact, for each vertex
$x$ with $\theta_a(x)\leq 2\pi,$ there is a neighborhood of $x$ in
$S_a(G)$ which is isometric to a neighborhood of a pole in the
$1$-suspension of a circle of length $\theta_a(x),$ and the result
follows from \cite[Theorem~10.2.3]{MR1835418}. We denote by
$$\PCS:=\{G=(V,E,F): \mathrm{\ there\ exists\ } a>0 \mathrm{ \ such\
that\ } S_a(G)\in \Al\}$$ the class of finite planar graphs whose
spherical polyhedral surfaces have the curvature bounded below by $1$ in
the Alexandrov sense. We will prove that $$\PCS=\PC,$$ see
Proposition~\ref{prop:pcspc}. This suggests a possible way to study the
class $\PCS$ by known results on the class $\PC,$ where the latter
refers to the Euclidean setting.

In a planar graph, the pattern of a vertex $x$ is defined as a vector $$\pt(x):=(\mathrm{deg}(\sigma_1),\mathrm{deg}(\sigma_2),\cdots,\mathrm{deg}(\sigma_{N})),$$ where $\{\sigma_i\}_{i=1}^N$ are the faces incident to $x$ ordered by $\mathrm{deg}(\sigma_1)\leq\mathrm{deg}(\sigma_2)\leq\cdots\leq\mathrm{deg}(\sigma_{N}),$ and $N=\deg(x).$

\begin{definition}\label{def:critical side-length}For any vertex $x$ of pattern  $(\mathrm{deg}(\sigma_1),\mathrm{deg}(\sigma_2),\cdots,\mathrm{deg}(\sigma_{N}))$ with positive combinatorial curvature, we define \emph{the critical side-length of the vertex} $x$ (or for the pattern) by
 \begin{eqnarray}a_c(x)&=&a_c(\mathrm{deg}(\sigma_1),\mathrm{deg}(\sigma_2),\cdots,\mathrm{deg}(\sigma_{N}))
  \nonumber \\
 &:=&\max\left\{a\in\left[0,\frac{2\pi}{\mathrm{deg}(\sigma_{N})}\right]:
 K_a(x)\geq 0\right\}. \label{def:ac} \end{eqnarray}
For a planar graph $G\in\PCS$, \emph{the critical side-length of $G$} is defined as
 $$a_c(G):=\min_{x\in V}a_c(x).$$
\emph{The critical area of $G$} is defined as
$$\area^{\cri}(G):=\area_{a_c(G)}(G).$$ 
\end{definition}

Clearly $\area_a(G)$ is monotonely increasing in $a$. So, $$\area^{\cri}(G)=\max\{\area_a(G):S_a(G)\in \Al\}.$$ 

We say that a planar graph $G$ \emph{admits a spherical tiling with regular spherical polygons} if there is a spherical tiling with regular spherical polygons whose planar graph structure is isomorphic to $G,$ and denote by $\ST$ the set of such planar graphs. One is ready to prove the following proposition.

\begin{proposition}\label{prop:sphere}For a planar graph $G,$ the following are equivalent:
\begin{enumerate}
\item $G\in \ST.$ 
\item There is some $a>0$ such that $S_a(G)$ is isometric to the unit sphere, i.e. $K_a(x)=0$ for all $x\in V.$
\item $a_c(x)=a_c(y)$ for all $x,y\in V,$ and $K_{a_c(x)}(x)=0$ for all $x\in V.$
\item $\area^{\cri}(G)=4\pi.$
\end{enumerate}
\end{proposition}

For any $G\in \ST,$ consider the spherical tiling in the unit sphere. We
observe that the convex hull of the vertices in $\R^3$ is a convex
polyhedron with regular Euclidean polygons as faces.  Combining the
above proposition with the classification of such convex polyhedra
\cite{MR0185507,Zalgaller69,MR0227860}, we classify the class $\ST$ in
the following theorem. 
 \begin{theorem}\label{thm:classification} The set  $\ST$ consists of the following:
\begin{enumerate}[(a)]
\item the {5} Platonic solids,
\item the 13 Archimedean solids, 
\item the infinite series of prisms and antiprisms,
\item \label{list:johnson}22 Johnson solids $J1,$ $J3,$  $J6,$ $J11,$ $J19,$
$J27,$ $J34,$  $J37,$ $J62,$ $J63,$
$J72,$ $J73,$  $J74,$ $J75,$ $J76,$
$J77,$ $J78,$  $J79,$ $J80,$ $J81,$ $J82,$ $J83.$
\end{enumerate}
 \end{theorem}
\begin{remark}\label{rem} The Johnson solids in the list \eqref{list:johnson} coincide with the
 Johnson solids with a circumsphere containing all vertices, except
 Johnson solids
 $J2,$ $J4,$ and $J5,$ see Wikipedia~\cite{wikijohnson}. $J2$~$(J4$,
 $J5$, resp.$)$ is excluded from the list~\eqref{list:johnson}  by Proposition~\ref{prop:sphere}, since it has $5$~$(8, 10$, resp.$)$ vertices $x$'s of pattern
 $(3,3,5)$~$((3,4,8)$, $(3,4,10)$, resp.$)$ such that
 $K_{a_c(x)}(x)>0$. Actually, such $x$'s are exactly the vertices of the
 unique $5$~$(8, 10$, resp.$)$-gonal face $F$ of $J2$~$(J4$, $J5$,
 resp.$)$, where $F$ is not contained in a hemisphere. {So $F$ is
 not a regular spherical polygon.}
\end{remark}
In \cite{MR929249}, Milka proved that in spherical three-dimensional
space there exists only finite number of (combinatorial types of) convex polyhedra with
equiangular faces, except two infinite families -- prisms and
antiprisms. However, he did not use the critical side-length $a_c$.
For other classification results on spherical
tilings, one refers to \cite{MR886775,MR3022611,MR661770}, for example.
For a combinatorial approach to classify tilings of constant curvature
spaces, see \cite{MR825773}.

We are interested in critical areas of planar graphs in $\PCS$. In particular, we propose the following problem.
\begin{problem}\label{prob:minarea} What are the following constants:
 \begin{enumerate}
 \item
$\area_{\min}:=\inf_{G\in \PCS} \area^{\cri}(G),$ and
\item 
$\wt{\area_{\max}}:=\sup_{G\in \PCS\setminus \ST} (\area^{\cri}(G))?$
\end{enumerate}
\end{problem}

Both constants in the above problem are attained by some specific graphs
in $\PCS,$ by Theorem~\ref{thm:208}~\eqref{assert:208}, since all prisms and antiprisms
admit spherical tilings with regular spherical polygons and have critical area $4\pi.$ In this paper, we will give quantitative bounds for the constants in the problem. 

The first part of Problem~\ref{prob:minarea} is devoted to the minimal critical area in the class $\PCS.$ We prove the following theorem. 
\begin{theorem}\label{thm:minarea}$8.3755\times 10^{-2}\leq \area_{\min}\leq 2.0961\times 10^{-1}.$
\end{theorem}

For the upper bound estimate in the theorem, we provide a graph in $\PCS$ with small critical area. For the lower bound estimate, we use the following local argument: For any graph $G\in \PCS,$ we consider the local structures of vertices which attain the critical side-length of the graph. We bound the critical area of the graph below by the summation of areas of faces incident to such a vertex, and compute the results case by case. This is amenable by  the result of Ghidelli, Theorem~\ref{thm:208}~\eqref{thm:41upper}, which considerably reduces the cases of possible vertex patterns, up to the facial degree $41$.

Next, we consider the second part of Problem~\ref{prob:minarea}. By Proposition~\ref{prop:sphere}, $$\area^{\cri}(G)=4\pi,\quad \mathrm{for\ any\ } G\in \ST,$$ which attains the maximal critical area in the class $\PCS,$ by the volume comparison theorem for Alexandrov spaces with curvature at least $1,$ see \cite{MR1835418}. The problem is to determine the maximum of critical areas in the class $\PCS$ except spherical tilings $\ST.$ We prove the following quantitative result.

\begin{theorem}\label{thm:gap}$4\pi-2.5678\times {10^{-1}}\leq \wt{\area_{\max}} \leq 4\pi-1.6471\times {10^{-5}}.$
\end{theorem}

The lower bound estimate follows from the calculation of critical area
of Johnson solid, $J16.$ For the upper bound estimate, one observes that
only local arguments as in the proof of Theorem~\ref{thm:minarea} are
insufficient. Our strategy is to reformulate the problem to a new one
which can be estimated by local arguments.  One is ready to see that
$$\wt{\area_{\max}}=4\pi-\epsilon_{\mathrm{gap}},$$ where
$$\epsilon_{\mathrm{gap}}:=\inf_{G\in \PCS\setminus \ST}
(4\pi-\area^{\cri}(G)),$$ which is called \emph{the gap} between the
maximal critical area $4\pi$ and other critical areas. Hence the above
theorem is equivalent to the gap estimate, $$1.6471\times {10^{-5}}\leq
\epsilon_{\mathrm{gap}} \leq 2.5678\times {10^{-1}}.$$ By the
Gauss-Bonnet theorem, \eqref{GaussBonnet}, we have
\begin{align}
 \epsilon_{\mathrm{gap}}=\inf_{G\in \PCS\setminus \ST} \sum_{\mbox{$x$
 is a vertex of $G$}}K_{a_c(G)}(x).\label{total angle defect}
\end{align}
Hence for the upper bound
estimate, it suffices to obtain the lower bound estimate of total angle
defect for these planar graphs. This new problem fits to local
arguments, and we prove the results by enumerating all cases.

The following is a corollary of Theorem~\ref{thm:gap}.
\begin{corollary}  $$G\in\PCS\setminus\ST,\ 0<a\leq a_c(G) \implies \area_a(G)\leq 4\pi-1.6471\times {10^{-5}}.$$
\end{corollary} 


\section{Preliminaries}
Let $G=(V,E,F)$ be a finite planar graph. Two vertices are called
\emph{neighbors} if there is an edge connecting them. We denote by $\deg(x)$
the degree of a vertex $x,$ i.e. the number of neighbors of a vertex
$x,$ and by $\deg(\sigma)$ the degree of a face $\sigma,$ i.e. the
number of edges incident to a face $\sigma$ (equivalently, the number of
vertices incident to $\sigma$). 


For any spherical regular $n$-gon of side length $a<\frac{2\pi}{n},$
$\Delta_n(a)$, contained in a hemisphere of $\SP^2(1),$ we denote by $\beta=\beta_{n,a}$ the
interior angle at corners of the $n$-gon. By the spherical cosine law
for angles~\cite[p.~65]{MR1254932},
\begin{equation*}\cos\frac{a}{2}\sin\frac{\beta}{2}=\cos\frac{\pi}{n}.
\end{equation*} One is ready to see that $\beta_{n,a}$ is monotonely increasing in $a$ and
\begin{equation}\label{eq:theta}\lim_{a\to 0}\beta_{n,a}=\frac{n-2}{n}\pi,\end{equation} where the right hand side is the interior angle of a regular $n$-gon in the plane.
The area of $\Delta_n(a)$ is given by 
$$\area(\Delta_n(a))=n\beta-(n-2)\pi.$$

We are ready to prove the Gauss-Bonnet theorem for spherical polyhedral surfaces.
\begin{proof}[Proof of Theorem~\ref{thm:GaussBonnet}] For any face $\sigma\in F,$ the area of $\sigma$ in $S_a(G)$ is given by
$$\deg(\sigma)\beta_{\deg(\sigma),a}-(\deg(\sigma)-2)\pi.$$ Hence by the counting argument,
\begin{eqnarray*}\area_a(G)&=&\sum_{\sigma\in F}\deg(\sigma)\beta_{\deg(\sigma),a}-\sum_{\sigma\in F}(\deg(\sigma)-2)\pi\\
&=&\sum_{x\in V}\theta_a(x)-\pi\sum_{\sigma\in F}\deg(\sigma)+2\pi\sharp F\\
&=&-\sum_{x\in V}K_a(x)+2\pi(\sharp V-\sharp E+\sharp F)=-\sum_{x\in V}K_a(x)+4\pi.
\end{eqnarray*} This proves Theorem~\ref{thm:GaussBonnet}. 
\end{proof}

\begin{proposition}\label{prop:pcspc} $\PCS=\PC.$
\end{proposition}
\begin{proof} For any $G\in \PCS,$ there is $a>0$ such that $S_a(G)\in \Al.$ For any face $\sigma\in F$ with $\deg(\sigma)=n,$ we know that the interior angle of $\Delta_n(a)$ is greater than that of {an} Euclidean $n$-gon. Consider the Euclidean polyhedral surface $S(G).$ Note that the total angle at each vertex in $S(G)$ is less than that in $S_a(G).$ This yields that $G\in \PC.$ This proves that $\PCS\subset \PC.$

For the other direction, let $G\in \PC.$ Note that by \eqref{eq:theta}, for each vertex $x\in V,$ there is a small constant $a(x)$ such that the total angle $\theta_{a(x)}(x)<2\pi.$ Since the graph is finite, we can choose a small constant $a$ such that
 $$S_a(G)\in \Al,$$ which proves that $G\in \PCS.$
\end{proof}

For a vertex $x$ of the pattern $(f_1,f_2,\cdots,f_N),$ with $N=\deg(x)$ and $\{f_i\}_{i=1}^N$ are the degrees of faces incident to $x.$ Then for any $0<a\leq \frac{2\pi}{f_N},$ the total angle at the vertex $x$ in $S_a(G)$ for some graph $G$ is given by
\begin{equation}\label{eq:ta1}\theta_a(x)=\sum_{i=1}^N2\arcsin \frac{\cos{\frac{\pi}{f_i}}}{\cos{\frac{a}{2}}}.\end{equation}
To determine the critical side-length of the vertex, we have two cases.
If $\theta_{\frac{2\pi}{f_N}}(x)<2\pi,$ then $a_c(x)=\frac{2\pi}{f_N}.$
Otherwise, $a_c(x)$ is the unique solution to the following equation
\begin{equation}\label{eq:ta2}\theta_a(x)=2\pi.\end{equation}

Now we prove Proposition~\ref{prop:sphere}.
\begin{proof}[Proof of Proposition~\ref{prop:sphere}] $(1)\Longrightarrow (2):$ This is trivial.

$(2)\Longrightarrow (3):$ This follows from the monotonicity of $\theta_a(x)$ in $a$ for any $x\in V.$

$(3)\Longrightarrow (4):$ Since $S_{a_c(G)}(G)$ has the curvature at least $1$ and is smooth at each vertex, hence it is locally isometric to a domain in $\SP^2(1).$ This implies that $S_{a_c(G)}(G)$ is isometric to $\SP^2(1).$ Hence $\area^{\cri}(G)=4\pi.$

$(4)\Longrightarrow (1):$ We know that $S_{a_c(G)}(G)$ has the curvature at least $1$ and $\area^{\cri}(G)=4\pi.$ Hence the rigidity of Bishop-Gromov volume comparison for Alexandrov surfaces with the curvature at least $1$ yields that $S_{a_c(G)}(G)$ is isometric to $\SP^2(1),$ see e.g. \cite[Exercise~10.6.12]{MR1835418}.
\end{proof}

Next, we give the proof of Theorem~\ref{thm:classification}.
\begin{proof}[Proof of Theorem~\ref{thm:classification}]
For any $G$ admitting a spherical tiling with regular spherical polygons, $S_{a_c(G)}(G)$ is the spherical tiling. Consider the convex hull $A$ of the vertex set $V$ in $S_{a_c(G)}(G)$ in $\R^3.$ We obtain a convex polyhedron $A$ in $\R^3$ such that all faces are regular Euclidean polygons. The classification of such polyhedra in $\R^3$ was obtained by \cite{MR0185507,Zalgaller69,MR0227860}. Any convex polyhedron in $\R^3$ such that all faces are regular Euclidean polygons is one of the following:
\begin{enumerate}[(a)]
\item the {5} Platonic solids,
\item the 13 Archimedean solids, 
\item the infinite series of prisms and antiprisms, and
\item the 92 Johnson solids.
\end{enumerate}
It is easy to see that the examples in $(a),(b),(c)$ admit spherical
 tilings with regular spherical polygons.
 For example, the antiprism of $2n$~($n>3$) vertices admits a spherical tilings
 with $2n$ regular triangles and $2$ regular $n$-gons having the vertices $P_i$
 $(i=0,1,\ldots,2n-1)$ such that (1) the distance between $P_i$ and the north pole is $\pi /2+ \left( -1 \right) ^{i}\arctan \left( \frac{1}{2}\,\sqrt {2-4\, 
\cos^2  {\frac {\pi }{n}}   +2\,\cos {
\frac {\pi }{n}}  } \right) 
$, (2) the longitude of $P_i$ is $\pi i/n$, and (3) $P_i$ is adjacent to
 $P_{i+1\bmod 2n}$ and to $P_{i+2\bmod 2n}$. The side-length for $n=5$
 is indeed the side-length $\arctan2$ of the regular spherical
 icosahedron on the unit sphere. These are all computed by a mathematics
 software \texttt{Maple}.
To complete the classification of spherical tilings with regular
 spherical polygons $\ST,$ it suffices to check Johnson solids case by
 case using the property $(3)$ in Proposition~\ref{prop:sphere}{, see Table~\ref{tbl:circu} for some part of results.}
 \begin{table}[ht]\centering
   \begin{tabular}{|l|l|l|}
     \hline
    Johnson Solid & the vertex patterns & $a_c$\\
    \hline
   $J1$  & $(3,3,3,3),(3,3,4) $  & $\pi/2$ \\
   $J3$  & $(3,3,4,4),(3,4,6)$   & $\pi/3$ \\
    $J6$  & $(3,3,5,5), (3,5,10)$ & $\pi/5$ \\
    $J11$ & $(3,3,3,3,3), (3,3,3,5)$  & $\arctan(2)$ \\ 
$J19$ & $(3, 4, 4, 4),(4, 4, 8)$ & $2\arccos \frac{\sqrt {12-2\sqrt2} \left( 1+3\sqrt{2}\right)}{17}$ \\
    $J62$ & $\begin{array}{l}
(3,3,3,3,3), (3,3,3,5),     \\ (3,5,5)
	     \end{array}$  & $\arctan(2)$ \\
   ${\begin{array}{l}Jn\\ (76\le n \le 83)\end{array}}$ & $(3,4,4,5), (4,5,10) $
	&$2\arccos{\frac {\sqrt {75-10\sqrt {5}}\sqrt {2} \left(
	    2\sqrt {5}+15 \right) }{205}}$\\
    \hline
   \end{tabular}
  \caption{The critical side-length $a_c$ of the Johnson solids having a circumsphere containing all
  vertices and more than two vertex
  patterns. See Theorem~\ref{thm:classification}
  and  Remark~\ref{rem}. \label{tbl:circu}}
 \end{table}
This proves Theorem~\ref{thm:classification}.\end{proof}

By Theorem~\ref{thm:208}~\eqref{thm:41upper}, for any graph in $\PC$ the
maximal degree of faces is at most $41.$ Motivated by this result, we
say a vertex pattern $(f_1,f_2,\cdots,f_{N})$ is \emph{admissible} if it
has positive combinatorial curvature and $f_N\leq 41.$ We denote by
$\ad$ the set of admissible vertex patterns. The set $\ad$ consists of
the following 342 tuples, derived from the list of vertex patterns with positive combinatorial curvature~\cite{MR2299456,MR2410938}.
\begin{table}[ht]
\begin{center}
\begin{tabular}{ |l|l|l| }
  \hline
  $(3,3,k),$  $3\leq k\leq 41$ &$(3,10,k),$ $10\leq k\leq 14$&$(5,6,k),$ \ \ \ \ \ \ $6\leq k\leq 7$\\
 $(3,4,k),$   $4\leq k\leq 41$ & $(3,11,k),$ $11\leq k\leq 13$& $(3,3,3,k),$ \ \ \ $3\leq k\leq 41$\\
   $(3,5,k),$  $5\leq k\leq 41$ & $(4,4,k),$\ \ \ $4\leq k\leq 41$& $(3,3,4,k),$ \ \ \ $4\leq k\leq 11$\\
 $(3,6,k),$ $6\leq k\leq 41$ & $(4,5,k),$\ \ \  $5\leq k\leq 19$& $(3,3,5,k),$ \ \ \ $5\leq k\leq 7$\\
  $(3,7,k),$ $7\leq k\leq41$ &$(4,6,k),$\ \ \  $6\leq k\leq 11$& $(3,4,4,k),$  \ \ \  $4\leq k\leq 5$\\
  $(3,8,k),$ $8\leq k\leq 23$ &$(4,7,k),$\ \ \  $7\leq k\leq 9$& $(3,3,3,3,k),$ $3\leq k\leq5$\\
  $(3,9,k),$ $9\leq k\leq 17$ & $(5,5,k),$\ \ \  $5\leq k\leq 9$& \\

  \hline
\end{tabular}
\end{center}
\caption{{Admissible patterns $\ad$.}\label{tbl:ad}} 
\end{table}
Table~\ref{tbl:ad}  is obtained from \cite[Table~1]{MR929249} under the restriction \cite[Table~3]{MR929249},
by ignoring the order of numbers in each tuple.

We can calculate critical side-length for all admissible vertex patterns
by using the formulas \eqref{eq:ta1} and \eqref{eq:ta2}. {One is ready to prove the following:} 
 { \begin{proposition}\label{prop:maxc}
   \begin{align*}
\max_{p\in \ad}a_c(p)=\arccos\frac{-1}{3}.\end{align*} The maximum is attained
  only at the pattern $p=(3,3,3).$
 \end{proposition}}
 In fact, $\arccos(-1/3)$ is the side-length of the spherical tiling by
 4 congruent regular triangles.

\section{Proof of Theorem~\ref{thm:minarea}}
In this section, we consider the estimates for the minimal critical area in the class $\PCS$ and prove Theorem~\ref{thm:minarea}.

For the upper bound estimate of the minimal critical area in the class
$\PCS,$ we consider an example in the class which has small critical
area, see Figure~\ref{example-min}. This example was initially
constructed by Ghidelli~\cite{Gh17} to show that there exists a graph
{$\Gamma\in\PC$} having a vertex of pattern $(3,7,41)$. Actually, 
the set of vertices in $\Gamma$ consists of 2 vertices $u$ of pattern
$(3,5,7)$, 6 vertices $v$ of pattern $(3,3,5,5)$, 53 vertices $w$ of pattern
$(3,3,5,7)$, 8 vertices $x$ of pattern $(3,5,41)$, 11 vertices $y$ of pattern
$(3,3,3,41)$, and 22 vertices $z$ of pattern $(3,7,41)$.
 By the numerical computation,
 \begin{align*}
a_c(\Gamma)
  =& \min (a_c(u),a_c(v),  a_c(w), a_c(x), a_c(y),  a_c(z))\\
=&\min(.86961\cdots,\ \pi/5,\ .22634\cdots,\
2\pi/41, \ .15291\cdots,\
  .030382\cdots)\\
 =&a_c(z)= .030382\cdots=:a_0.
  \end{align*}
The set of faces in $\Gamma$ consists of 61 triangles, 15 pentagons, 11
heptagons and a 41-gon. Hence
\begin{eqnarray*}\area^{\cri}(\Gamma)&=&61\area(\Delta_3(a_0))+15\area(\Delta_5(a_0))\\&&+11\area(\Delta_7(a_0))+\area(\Delta_{41}(a_0))\\&=&2.0961\times
10^{-1}.\end{eqnarray*}
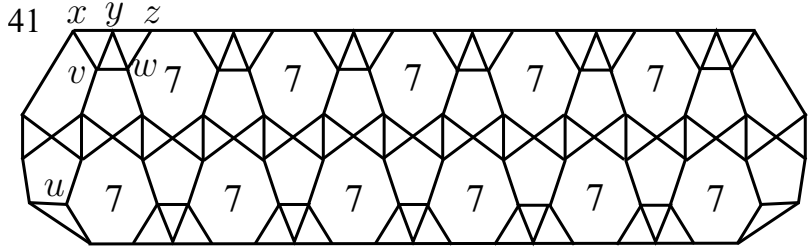
\begin{figure}[htbp]
 \begin{center}
   \begin{tikzpicture}
    \node at (0,0){\includegraphics[width=0.84\linewidth]{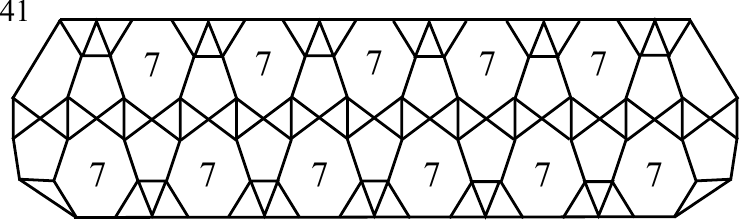}};
    \node at (-4.7, -.8){\Large $u$};
    \node at (-4.4,   .7){\Large $v$};
    \node at (-3.5,   .8){\Large $w$};
    \node at (-4.4,  1.5){\Large $x$};
    \node at (-3.9,  1.5){\Large $y$};
    \node at (-3.4,  1.5){\Large $z$};
   \end{tikzpicture}
  \caption{\small A planar graph $\Gamma\in\PCS$ has small critical area.
  $\pt(u)=(3,5,7)$,    $\pt(v)=(3,5,5,5)$,  $\pt(w)=(3,3,5,7)$, $\pt(x)=(3,5,41)$, $\pt(y)=(3,3,3,41)$, and $\pt(z)=(3,7,41)$.}
\label{example-min}
 \end{center}
\end{figure} 

For the lower bound estimate of the minimal critical area, we shall
estimate the area from below for any graph $G$ in the class $\PCS.$ For
the graph $G$, let $x_0$ be a vertex such that $a_c(x_0)=a_c(G).$ It is
obvious that $$\area^{\cri}(G)\geq
\sum_{x_0\in\overline{\sigma_i}}\area_{a_c(x_0)}(\sigma_i).$$ That is, we give
the lower bound estimate via the summation of areas of faces incident to $x_0.$
 So that
\begin{equation}\label{eq:min1}\area_{\mathrm{min}}\geq \min_{\pt(x)\in
\ad}\sum_{x\in\overline{\sigma_i}}\area_{a_c(x)}(\sigma_i).\end{equation} By the
numerical computation, we enumerate all critical side-lengths for
admissible vertex patterns. For each admissible vertex pattern, we
calculate the summation of areas of faces in the pattern using the
critical side-length of this pattern, as in \eqref{eq:min1}. Then the
minimum of calculated results gives the lower bound of the minimal
critical area, $8.3755\times 10^{-2},$ see the Appendix for the
algorithms~(Section~\ref{appendix:1.9}).

\section{Proof of Theorem~\ref{thm:gap}}
In this section, we estimate the first gap $\epsilon_{\mathrm{gap}}$ and prove Theorem~\ref{thm:gap}.

For the lower bound estimate of Theorem~\ref{thm:gap}, we calculate the
critical area of Johnson solid, $J16.$  The set of vertices in $J16$
consists of 10 vertices $x$ of pattern $(3,3,4,4)$ and 2 vertices $y$ of
pattern $(3,3,3,3,3).$ By the numerical computation,
\begin{eqnarray*}a_c(J16)=\min(a_c(x),\ a_c(y))=\min(1.0472,\ 1.1071)
=a_c(x)=1.0472=:a_0.\end{eqnarray*} Since the set of faces in $J16$ contains $10$ triangles and $5$ squares, 
$$\area^{\cri}(J16)=10\area(\Delta_3(a_0))+5\area(\Delta_4(a_0))=4\pi-2.5678\times {10^{-1}}.$$
Actually, $\area^{\cri}(J16)\ge\area^{\cri}(J)$ for any Johnson solid
$J$, according to a numerical computation.

For the upper bound estimate of Theorem~\ref{thm:gap}, we are based on
 \eqref{total angle defect} and the finiteness of $\PC\setminus\ST$.
 We estimate total
 angle defect $\sum_{x\in G}K_{a_c(G)}(x)$ from below, for any $G=(V,E,F)\in\PC\setminus\ST$.
 For the graph $G$, let $x_0$ be a vertex such that $a_c(x_0)=a_c(G),$ i.e. for any $y\in V,$ $a_c(y)\geq a_c(x_0).$ For simplicity, we denote by $a_0:=a_c(G).$ By Proposition~\ref{prop:maxc}, $\pt(x_0)\neq (3,3,3)$ since there are at least two different critical side-lengths for vertices in $G$ by $(3)$ in Proposition~\ref{prop:sphere}.
Now we have two cases: $K_{a_0}(x_0)=0$ and $K_{a_0}(x_0)>0$.

Consider the case $K_{a_0}(x_0)=0.$
 Since $\area^{\cri}(G)<4\pi,$ by the Gauss-Bonnet theorem
 \eqref{GaussBonnet}, $\sum_{x\in V}K_{a_0}(x)>0$.
 Hence, for some vertex $y\in V$,
 \begin{align*}K_{a_0}(y)>0,\quad
 \mathrm{and}\ \ a_c(y)\geq a_0.\end{align*}
Therefore, we obtain a positive lower bound of $\sum_{x\in V}K_{a_0}(x)$:
\begin{align}
\begin{cases}K_{a_0}(x_0),& (K_{a_0}(x_0)>0);\\
\min\left\{K_{a_0}(y)>0\ \colon\ y\in V, \ a_c(y)\geq a_c(x_0) \right\},&(K_{a_0}(x_0)=0).\end{cases}\label{ineq:pos}\end{align}

By replacing vertices $x,y$ by vertex patterns $p,q$ in the lower
bound~\eqref{ineq:pos} and then minimizing it over
$p\in\ad\setminus\{(3,3,3)\}$, we obtain the lower bound
$1.6471\times {10^{-5}}$ given in Theorem~\ref{thm:gap}.
 See the Appendix for the
algorithm~(Section~\ref{appendix:1.10}).

\section{Appendix}

We prove Proposition~\ref{prop:maxc} in Subsection~\ref{subsec:propmaxc}, and present
symbolic computations and numerical computations in the proof of
Theorem~\ref{thm:minarea} and
Theorem~\ref{thm:gap}.
The computations are carried out by a mathematics software \texttt{Maple}.

\subsection{Proof of  Proposition~\ref{prop:maxc}}\label{subsec:propmaxc}

  For all $p=(f_1, f_2, \ldots, f_n)\in\ad$ and $q=(g_1, g_2, \ldots,
  g_m)\in\ad$, we write  $p \le_{emb} q$ if there are
  positive integers  $1\le j_1<j_2<\cdots<j_n\le m$ such that $f_i\le g_{j_i}$
  $(1\le i\le n)$. In this case, $f_n\le g_m$, because $f_i \le f_j$ and $g_i
  \le g_j$ for  $i<j$.
  The binary relation $\le_{emb}$ is a partial
  ordering on $\ad$. {Let $<_{emb}$ be the strict part $\le_{emb}\setminus =$.}

 \begin{lemma}\label{lem:monotonicity} Assume
	 $p=(f_1,\ldots,f_n)\in\ad$ and {$q=(g_1,\ldots, g_m)\in\ad.$}
  \begin{enumerate}
   \item \label{assert:curvature defined} $K_a(p)$ is
	 a strictly decreasing, continuous function of $a$:
	 \begin{align*}a\in [0,\
	 2\pi/f_n]\mapsto K_a(p)=2\pi - \sum_{i=1}^n 2\arcsin\frac{\cos\frac{\pi}{f_i}}{\cos\frac{a}{2}}.\end{align*}

 \label{assert:decreasing in a} 
	 
   \item \label{assert:decreasing in p} If $p<_{emb} q$ and  $a\in
	 [0,\ 2\pi/g_m]$, then $K_a(p)> K_a(q)$.
  \end{enumerate}
 \end{lemma}
  \begin{proof}
    As $a\in [0,\ 2\pi/f_n]$ and $f_i\ge3$,
we have   $\pi/f_i\ge \pi/f_n\ge a/2$. So, 
   $\arcsin(\cos(\pi/f_i)/\cos(a/2))$ is defined. The other assertions
   are clear.   
  \end{proof}

\begin{lemma}\label{lem:critical side-length}
If $p=(f_1,\ldots,f_n)\in\ad$,
 $a_c(p)$ of Definition~\ref{def:critical side-length} is well-defined.
 Moreover,  if $K_{2\pi/f}(p)> 0$, then $a_c(p)=2\pi/f_n$. If $K_{2\pi/f}(p)\le 0$, then $a_c(p)$ is the unique solution
 $a$ such that $K_a(p)=0$ and $a\in (0,\
  2\pi/f_n{]}$.
\end{lemma}

\begin{proof} By  Lemma~\ref{lem:monotonicity}~\eqref{assert:decreasing
 in a}, the argument $U$ of {the maximum} in \eqref{def:ac} is a compact set. {By the definition of admissible patterns and the limiting behavior in \eqref{eq:theta}, one is ready to see that $a_c(p)=\max U>0.$} The latter part is due to
  Lemma~\ref{lem:monotonicity}~\eqref{assert:decreasing in a}.
\end{proof}

\begin{lemma}\label{lem:critical side-length monotone}
{$p<_{emb} q \implies a_c(p)> a_c(q)$.}
\end{lemma}
   \begin{proof} Let $p=(f_1,\ldots,f_n)$ and $q=(g_1,\ldots,g_m)$. Then
$     a_c(p)=\max\{a\in [0,\ 2\pi/f_n]\colon K_a(p)\ge0\}\ge \max\{a\in
    [0,\ 2\pi/g_m]\colon K_a(p)\ge0\}$ by $f_n\le g_m$. This is greater
    than or equal to $\max\{a\in [0,\ 2\pi/g_m]\colon
    K_a(q)\ge0\}=a_c(q)$ by
    Lemma~\ref{lem:monotonicity}~\eqref{assert:decreasing in
    p}. {Moreover, $a_c(p)>a_c(q)$ by Lemma~\ref{lem:monotonicity}.}
  \end{proof}

\begin{proof}[Proof of Proposition~\ref{prop:maxc}] Among $\ad$, the pattern
 $(3,3,3)$ is 
 the least with respect to {the partial ordering} $\le_{emb}$. By  Lemma~\ref{lem:critical
 side-length monotone}, the conclusion follows.
\end{proof}

\subsection{An algorithm to compute the lower bound of $\area_{\mathrm{min}}$ of
Theorem~\ref{thm:minarea}}\label{appendix:1.9}

We will compute
{\begin{align*}
\min_{(f_1,\ldots,f_n)\in\ad}\sum_{i=1}^n \area(\Delta_{f_i}(a_c(p))). 
\end{align*}}
To improve the accuracy of numerical computation of the critical side-length $a_c(p)$
of vertex patterns $p\in\ad$, we enumerate
\begin{align*}
 M:=\{\ (f_1,\ldots,f_n)\in\ad\mid K_{2\pi/f_n}(f_1,\ldots,f_n)> 0 \}.
\end{align*} 

 \begin{lemma} \label{lem:N}$M$ consists of the following 103 patterns:
$(3,3,k)$ $(5\le k\le41)$,  $(3,4,k)$ $(7\le k\le 41)$, and $(3,5,k)$ $(11\le k\le 41)$.
\end{lemma}
  \begin{proof}
For any increasing sequence $(f_1,\ldots,f_n)$ of positive integers,
  \begin{align*}K_{2\pi/f_n}(f_1,\ldots,f_n)=\pi - 2
  \sum_{i=1}^{N-1}\arcsin\frac{\cos(\pi/f_i)}{\cos(\pi/f_n)}.\end{align*}
     Note that for any $(f_1,f_2,f_3)\in\ad$
  \begin{align}
  K_{2\pi/f_3}(f_1,f_2,f_3)\lessgtr 0\iff \cos^2\frac{\pi}{f_1}
   +\cos^2\frac{\pi}{f_2} \gtrless \cos^2\frac{\pi}{f_3}. \label{criterion}
  \end{align}
   
$M$ contains the 103 patterns,  by \eqref{criterion} and
   $K_{2\pi/4}(3,3,4)=K_{2\pi/6}(3,4,6)=K_{2\pi/10}(3,5,10)=0$. We will
   prove that $M$ is contained in the 103 patterns. For two increasing
   sequences $(f_1,\ldots,f_n)$ and $(g_1,\ldots,g_m)$ of positive
   integers, define a partial ordering
  \begin{align*}
   (f_1,\ldots,f_n) \sqsubseteq (g_1,\ldots,g_m):\iff 
   (f_1,\ldots,f_{n-1})\le_{emb} (g_1,\ldots,g_{m-1}), f_n\ge g_m.
\end{align*}
  Then,  
  \begin{align*}
   (f_1,\ldots,f_n) \sqsubseteq (g_1,\ldots,g_m) \implies   K_{2\pi/f_n}(f_1,\ldots,f_n) \ge K_{2\pi/g_m}(g_1,\ldots,g_m).
  \end{align*}
It is sufficient to show $K_{2\pi/f_n}(f_1,\ldots,f_N)\le 0$ for all $\sqsubseteq$-minimal
  $(f_1,\ldots,f_N)\in\ad\setminus M$. 
All such $(f_1,\ldots,f_N)$ are $(3, 6, 41), (4, 4, 41)$, and $(3, 3, 3, 41)$.

   By \eqref{criterion}, $K_{2\pi/41}(3,6,41), K_{2\pi/41}(4,4,41)<0$.
   The vertex pattern of the 41-gonal antiprism is $p=(3,3,3,41)\in\ad$. By
   Theorem~\ref{thm:classification}, the 41-gonal antiprism is in
   $\ST$. By Proposition~\ref{prop:sphere}, $K_{a_c(p)}(p)=0$. By
   Lemma~\ref{lem:monotonicity}~\eqref{assert:decreasing in a} and
   $a_c(p)\le 2\pi/41$, $K_{2\pi/41}(3,3,3,41)\le 0$.
  \end{proof}
For all 239 patterns $p=(f_1,\ldots,f_n)\in\ad\setminus M$, in order to
compute $a_c(p)$, we numerically solve equations $K_a(p)=0$ and $a\in
(0,\ 2\pi/f_n)$, by using a floating-point arithmetic solver
\texttt{fsolve} already built in \texttt{Maple}.  In \texttt{Maple}, the
number of digits carried in float is 10, unless otherwise we change. However, because of numerical error,
\texttt{fsolve} is not unable to solve the equation for $p=(3,5,9)$. {In
this case},
we use a symbolic solver \texttt{solve} built in \texttt{Maple} {, to
solve the equation $K_a(3,5,9)=0$}. {Then} we verify the output of \texttt{solve}, by computing the numerical
value of the output of \texttt{solve}, and plotting a monotone function
$2\pi - \sum_{i=1}^n 2\arcsin\left(\cos(\pi/f_i)/\cos(a/2)\right)$ over
$a\in (0,\ 2\pi/f_n)$.

We should compute the minimum of a finite list of numbers, under
 numerical error. To control the precision of computation of
 \texttt{Maple}, we will compute the second minimum, and the third
 minimum, and so on. These will be computed by a \emph{sorting
 algorithms}, which puts elements of the list in a certain order.
For
 the complexity and mathematical properties of various sorting algorithm, see \cite{MR784432}.

By the computation, 
{\begin{align}
 \min_{(f_1,\ldots,f_n)\in\ad}\sum_{i=1}^n
\area(\Delta_{f_i}(a_c(f_1,\ldots,f_n)))=8.3755\cdots\times 10^{-2} \label{ans}
\end{align}} which is achieved only by
$p=(3,11,13)$. {It is sufficiently distant from}
 the
second minimum $1.2823\cdots\times 10^{-1}$, which is achieved only by
$p=(3,7,41)$. 
 This gives
the lower bound of $\area_\mathrm{min}$ given in
Theorem~\ref{thm:minarea}.
 
\subsection{An algorithm to compute the upper bound of  $\wt{\area_{\mathrm{max}}}$ of  Theorem~\ref{thm:gap}}
\label{appendix:1.10}

Based on the inequality~\eqref{ineq:pos} and the discussion just after
it, we will compute the minimum element in the following finite set
  \begin{align}
   \bigcup_{\substack{ p\in\ad\\ p\ne(3,3,3)}}
\left(\begin{array}{ll}
&\{K_{a_c(p)}(p)\ \colon\  K_{a_c(p)}(p)>0\}\\
  \cup&
\left\{K_{a_c(p)}(q)>0\ \colon\ q\in \ad,  a_c(q)\geq a_c(p),  K_{a_c(p)}(p)=0 \right\}
  \end{array}\right)\label{lb:pos}\end{align}
and the set of $(p,q)$ with $K_{a_c(p)}(q)$ is the minimum.

{To decide the condition $K_{a_c(p)}(p)=0$, we have only to check $p\notin M$,
because of Lemma~\ref{lem:N}. This causes no numerical error. }


{Even if $K_{a_c(p)}(q)=0$, a numerical computation of $K_{a_c(p)}(q)$
 could result in a very small positive, erroneous value.  To circumvent} this, we use the following:
\begin{definition}\label{def:M}
Let  the set $Z$ of 9 2-sets $\{p,q\}$ {consist of $\{(3,4,5), (4,4,4)\}$ and any two vertex patterns of
  Johnson solid $Jn$ for some $n$ with $n=1,3,6,11,19,62$ or $76\le n\le
  83$.} See {Table~\ref{tbl:circu}.}
\end{definition}

\begin{lemma} \label{lem:M}
 For any $p\in \ad\setminus M$ and for any $q\in\ad$, we have
$K_{a_c(p)}(q)=0$, if
  $p=q$ or $\{p,q\}\in Z$.
\end{lemma}
   \begin{proof}
By Lemma~\ref{lem:N}, if $p=q$, then $K_{a_c(p)}(q)=0$.
Consider {the case that $\{p,q\}=\{(3,4,5), (4,4,4)\}\in Z$.}
Note that $a_c(4,4,4)=2\,\arccos \left( \sqrt
{6}/3 \right)$. We claim $K_{a_c(4,4,4)}(3,4,5)=0$. It is because
$$K_{a_c(4,4,4)}(3,4,5)=\frac{4\pi}{3} - 2 \arcsin\frac{\sqrt{6}}{4} -
2\arcsin\frac{\left( \sqrt {5}+1 \right) \sqrt {6}}{8}=0,$$ if and only if $$2/3\,\pi -\arcsin \left( \sqrt {6}/4
\right) =\arcsin \left(  \left( \sqrt {5}+1 \right) \sqrt {6}/8
\right),$$ which is equivalent to $$\sin \left( \pi /3+\arcsin \left(\sqrt
{6}/4 \right) \right) = \left( \sqrt {5}+1 \right) \sqrt {6}/8.$$ So
$K_{a_c(4,4,4)}(3,4,5)=0$. Because $K_a(3,4,5)$ is monotone, we have
$a_c(4,4,4)=a_c(3,4,5)$. So,
    $K_{a_c(3,4,5)}(4,4,4)=K_{a_c(4,4,4)}(4,4,4)=0$.
In the other case,  $K_{a_c(p)}(q)=0$ is by Theorem~\ref{thm:classification}.
\end{proof}

In a numerical computation of the minimum of the finite set~\eqref{lb:pos},
numerical errors of $K_{a_c(p)}$, $a_c(q)$ and $a_c(p)$ matter for deciding the
 two inequalities $K_{a_c(p)}(q)>0$ and $a_c(q)\ge a_c(p)$ in \eqref{lb:pos}.
 So, we will numerically compute
 the triples $(p,q,k)$ in the following set $S(\epsilon)$ such that the third
 entry $k$ is the
 first minimum or the second minimum among all $k'$ with $(p',q',k')\in S(\epsilon)$.
 \begin{align*}S(\epsilon)  &:=\{(p,p,2\pi/41)\colon p=(f_1,f_2,41) \in M\}\\
  &\cup
  \left \{(p,q,K_{a_c(p)}(q)) \colon\begin{array}{l}
    K_{a_c(p)}(q)>\epsilon,\ q\in \ad\setminus\{p\},\
  \{p,q\}\notin Z, \\
a_c(q)\geq\epsilon+ a_c(p),\ p\in \ad\setminus M \setminus\{(3,3,3)\}
\end{array}\right\},\end{align*} {where  $\epsilon$ of the two
inequalities  $K_{a_c(p)}(q)>\epsilon$ and $a_c(q)\ge\epsilon+ a_c(p)$ is any real number. }

{Note that} $$ |S(\epsilon)|\ge(\sharp M)+(\sharp
\ad - 1)\cdot (\sharp\ad  - \sharp
M - 1) - 2\sharp Z =81243.$$

As the ``margin'' $\epsilon\in\R$ of the inequalities decreases,
{the set $S(\epsilon)$ becomes large.}
A numerical computation shows that the smallest third entries of  $S(\pm
  10^{-5})$ are both $1.6471\cdots\times
  10^{-5}$ and {are} achieved only by $K_{a_c(3,7,29)}(3,9,16)$, while the
  second smallest third {entries} of $S(\pm 10^{-5})$ {are both} $1.7919\times
  10^{-5}$ and {are} achieved only by $K_{a_c(4, 4, 28)}(5, 5, 9)$. {The numerical computations are done by  \texttt{Maple}
  with the number of digits carried in float being 10.
  The smallest third entries $1.6471\cdots\times
  10^{-5}$ are sufficiently distant from the second smallest third entries $1.7919\times
  10^{-5}$. 
  So, the minimum of $S(0)=\eqref{lb:pos}$ is $1.6471\cdots\times
  10^{-5}$  and is achieved only by $K_{a_c(3,7,29)}(3,9,16)$.} Hence
  $$\wt{\area_{\max}}\le 4\pi-K_{a_c(3,7,29)}(3,9,16)=4\pi-1.6471\cdots\times 10^{-5}.$$
{This completes the proof of Theorem~\ref{thm:gap}.}

{The critical side-lengths of the argmin vertex patterns $p=(3,7,29), q=(3,9,16)$ are very close:}
   $0.14267\cdots=a_c(3,7,29)<a_c(3,9,16)=0.14269\cdots$ and there is no
$r\in\ad$ such that $a_c(3,7,29)<a_c(r)<a_c(3,9,16)$, according to a
numerical computation. 



\bigskip
{\textbf {Acknowledgements.}}\ We thank Andrei Vesnin for helpful discussions on Milka's work.
Y. A. is supported by JSPS KAKENHI Grant Number JP16K05247.  Y. S. is supported by  NSFC, grant
  no. 11771083 and NSF of Fujian Province through Grants 2017J01556,
  2016J01013.

\bibliographystyle{spmpsci}
\bibliography{sphericaltilings-v2}

\bigskip

\bigskip
\end{document}